\documentclass[11pt, reqno ]{amsart}
\usepackage{geometry}                
\usepackage{graphicx}
\usepackage{amssymb}
\usepackage{epstopdf}
\usepackage{amsmath}
\usepackage{mathrsfs}
\usepackage{tikz-cd}
\usepackage[all]{xy}
\usepackage{color}
\usepackage{tikz}
\usepackage{xcolor}
\usepackage{mathtools}
  \usepackage{enumitem}
  \usepackage[all]{xy}
\usepackage{color}
\usepackage{tikz}
  \usepackage{array}
  \usepackage{amsthm}

\usepackage[colorlinks=true, pdfstartview=FitV,linkcolor=blue,citecolor=blue,urlcolor=blue]{hyperref}

    \advance \topmargin by -\headheight
    \advance \topmargin by -\headsep

    \evensidemargin \oddsidemargin
\numberwithin{equation}{section}

\DeclareMathOperator{\Res}{Res}

\DeclareMathOperator{\Id}{id}
\newcommand{\C}{\mathbb C}

\newcommand{\slTwo}{\mathfrak{sl}_2}

\newcommand{\cS}{\mathcal S}

\newcommand{\Dr}{\mathscr D}

\newcommand{\ee}{\mathsf E}
\newcommand{\ff}{\mathsf F}
\newcommand{\hh}{\mathsf H}
\newcommand{\dd}{\partial}
\newcommand{\norm}[1]{\mathopen{:}#1\mathclose{:}}

\def\cS{{\mathcal S}}

\newfont{\german}{eufm10}

\begin{document}

\newtheorem{maintheorem}{Main Theorem}
\newtheorem{thm}[subsection]{Theorem}
 \newtheorem{theorem}{Theorem}[section]
  \newtheorem{prop}[theorem]{Proposition}
    \newtheorem{proposition}[theorem]{Proposition}
  \newtheorem{cor}[theorem]{Corollary}
    \newtheorem{corollary}[theorem]{Corollary}
  \newtheorem{lemma}[theorem]{Lemma}
  \newtheorem{defn}[theorem]{Definition}
  \newtheorem{ex}[theorem]{Example}
    \newtheorem{conj}[theorem]{Conjecture}
     \newtheorem{remark}[theorem]{Remark}
          \newtheorem{example}[theorem]{Example}

 \newcommand{\cx}{{\bf C}}
\newcommand{\la}{\langle}
\newcommand{\ra}{\rangle}
\newcommand{\res}{{\rm Res}}
\newcommand{\expp}{{\rm exp}}
\newcommand{\lgp}{\widehat{\rm G}( { F} ((t)) )}
\newcommand{\svee}{\scriptsize \vee}
\newcommand{\deff}{\stackrel{\rm def}{=}}
\newcommand{\cal}{\mathcal}

\newcommand{\on}{\operatorname}
\newcommand{\mc}{\mathcal}

\title{A new $SL_2$-action on the rank-two $\beta\gamma$-system}

\author{Tomoyuki Arakawa $^{1}$}
\email{TOMOYUKI.ARAKAWA@OIST.JP}

\author{Xuanzhong Dai $^{2}$}
\email{xzdai@nju.edu.cn}

\author{Bailin Song $^{3}$}
\email{bailinso@ustc.edu.cn}

\address{$^1$ Okinawa Institute of Science and Technology
1919-1 Tancha, Onna-son, Kunigami-gun
Okinawa, Japan 904-0495}

\address{$^2$ School of Mathematics, Nanjing University,
Nanjing 210093, China}
\address
{$^3$ School of Mathematical Sciences, University of Science and Technology of China, Hefei, Anhui 230026, P. R. CHINA}

\thanks{}

\maketitle

\maketitle

\begin{abstract}
We construct an explicit rational \(\mathrm{SL}_2(\mathbb C)\)-action
on the underlying vector space of the rank-two \(\beta\gamma\)-system.
The construction is motivated by Creutzig's branching decomposition
with respect to \(L_{-1}(\mathfrak{sl}_2)\) and a rank-one Heisenberg
vertex algebra. We realize the raising and lowering operators as
residues of lattice-type fields, prove that they preserve the original
\(\beta\gamma\)-system despite being defined on localizations, and
establish their commutation relations and integrability.
The resulting action commutes with the affine
\(L_{-1}(\mathfrak{sl}_2)\)-action and all nonzero Heisenberg modes,
while the Heisenberg zero mode serves as its Cartan operator.
Each affine multiplicity space is thereby identified with
\(V_n\otimes\mathcal F_{\mathrm{osc}}\), where \(V_n\) is the irreducible
\(\mathrm{SL}_2(\mathbb C)\)-module of highest weight \(n\) and
\(\mathcal F_{\mathrm{osc}}\) is the Heisenberg oscillator Fock space.
Thus the Heisenberg charge decomposition is realized as the weight
decomposition of \(V_n\). The infinitesimal action is not by vertex
algebra derivations and does not preserve the symmetric conformal
grading.
\end{abstract}

\maketitle

\section{Introduction}\label{sec:intro}

Free-field vertex algebras provide explicit realizations of affine
and $\mathcal W$-algebras and a setting in which their representation
theory can be studied concretely. Their branching rules often reveal
structures that are not apparent from the operator product expansions
of the generating fields. In particular, the multiplicity spaces
arising from a vertex subalgebra may suggest additional Lie algebra
actions. Constructing such actions explicitly requires more than a
character identity: one must identify operators on the original
state space and establish their commutation relations and
integrability. The purpose of this note is to carry out such a
construction for the rank-two $\beta\gamma$-system.

Our motivation also comes from the correspondence between
four-dimensional $\mathcal N=2$ superconformal field theories and
two-dimensional vertex algebras introduced in \cite{BLLPRvR}.
This correspondence equips a protected sector of local operators
with a vertex algebra structure and identifies the Schur index
with a vacuum character, up to the conventional normalization.
Free hypermultiplets give rise to symplectic bosons, or
$\beta\gamma$-systems with generators of conformal weight
$\frac12$. These free-field examples therefore provide a natural
setting in which to investigate the algebraic structures underlying
protected-state counting. In \cite{C}, Creutzig studied vertex
algebras associated with the $(A_1,A_{2p-3})$ and
$(A_1,D_{2p})$ Argyres--Douglas theories, relating their characters
to Schur indices. The rank-two $\beta\gamma$-system appears there
as the $p=1$ starting point for the analysis of the $D_{2p}$ family.

Let $\cS=\cS(\C^2)$ be the rank-two $\beta\gamma$-system generated by
$\beta_i,\gamma_i$, $i=1,2$, with operator product expansions
\begin{equation}\label{eq:bg}
 \beta_i(z)\gamma_j(w)\sim\frac{\delta_{ij}}{z-w},
 \qquad
 \gamma_i(z)\beta_j(w)\sim-\frac{\delta_{ij}}{z-w},
\end{equation}
all other pairs of generators having regular OPEs. The standard
bilinear currents generate a vertex subalgebra
$L_{-1}(\mathfrak{sl}_2)\otimes\pi$, where
$L_{-1}(\mathfrak{sl}_2)$ is the simple affine vertex algebra
at level $-1$ (see Section \ref{sec:creutzig-decomposition})
 and $\pi$ is the rank-one Heisenberg vertex algebra
generated by
\begin{equation}\label{eq:h}
 h=\norm{\beta_1\gamma_1}+\norm{\beta_2\gamma_2}
\end{equation}
with OPE
\[
 h(z)h(w)\sim-\frac{2}{(z-w)^2}.
\]
The zeroth mode $h_0$ acts semisimply on $\cS$, assigning charge
$1$ to each $\gamma_i$ and charge $-1$ to each $\beta_i$.

The immediate motivation for our construction is Creutzig's
branching formula \cite[Proposition~5.1]{C}. Let $L_{-1}(n)$ denote
the simple highest-weight affine $\mathfrak{sl}_2$-module at
level $-1$ whose top is the $(n+1)$-dimensional
$\mathfrak{sl}_2$-module, and let $\pi_s$ denote the Heisenberg
Fock module of $h_0$-charge $s$. Then
\begin{equation}\label{eq:creutzig-decomposition}
 \cS
 \cong
 \bigoplus_{n\geq0}
 L_{-1}(n)\otimes
 \left(\bigoplus_{j=0}^{n}\pi_{n-2j}\right)
\end{equation}
as an $L_{-1}(\mathfrak{sl}_2)\otimes\pi$-module.
For each $n$, the charges
\[
 n,\ n-2,\ \ldots,\ -n
\]
are precisely the weights of the irreducible
$\mathfrak{sl}_2$-module of highest weight $n$.
This pattern suggests an additional $\mathfrak{sl}_2$-action
relating the Heisenberg Fock summands in each affine multiplicity
space, with $h_0$ as its Cartan operator.

The question is therefore whether this pattern can be realized
by explicit raising and lowering operators on $\cS$.
We construct such operators using residues of lattice-type fields 
associated with Heisenberg exponentials. Although the construction initially
takes place in a localization of $\cS$, the resulting operators
preserve $\cS$ itself and integrate to an $SL_2(\C)$-action.
This action is an action on the underlying vector space, and 
its infinitesimal action is not by vertex algebra derivations.
Accordingly, the construction gives additional representation
theoretic structure on the state space, but does not by itself
establish an additional four-dimensional flavor symmetry.
It provides a concrete setting in which to examine the relation
between branching patterns and symmetries of operator products.

We now describe the construction. The assignments
\begin{equation}\label{eq:S}
 \begin{aligned}
 S(\gamma_1)&=\beta_2,
 &\qquad S(\gamma_2)&=-\beta_1,\\
 S(\beta_1)&=-\gamma_2,
 &\qquad S(\beta_2)&=\gamma_1
 \end{aligned}
\end{equation}
preserve the OPEs \eqref{eq:bg} and hence extend to a
vertex algebra automorphism $S$ of $\cS$. It satisfies
\[
 S^2=\Id,\qquad S(h)=-h,
\]
and consequently $Sh_0S=-h_0$.

For an even Heisenberg field $K$ and a field $a$, consider the
formal expression
\begin{equation}\label{eq:dressing-intro}
 \Dr_K(a;z)
 :=
 \exp\!\left(\sum_{m<0}\frac{K_m}{m}z^{-m}\right)
 a(z)z^{-K_0}
 \exp\!\left(\sum_{m>0}\frac{K_m}{m}z^{-m}\right),
\end{equation}
whenever it is well defined on the space under consideration.
We call $\Dr_K(a;z)$ the \emph{lattice-type field} associated
with $a$ and $K$, and its residue the corresponding
\emph{lattice-type residue}. In particular on the localization
$\cS[\gamma_1^{-1}]$, we set
\begin{equation}\label{eq:free-P}
 A=\norm{\gamma_1^{-1}\beta_2},
 \qquad
 P=\Res_z\Dr_h(A;z).
\end{equation}
Here is the main theorem of this note.

\begin{theorem}\label{thm:main}
The following statements hold.
\begin{enumerate}
 \item
 The operator $P$ defined in \eqref{eq:free-P} preserves $\cS$.
 Moreover, the operators
 \begin{equation}\label{eq:triple}
  \ee:=SPS,\qquad
  \ff:=P,\qquad
  \hh:=h_0
 \end{equation}
 satisfy the commutation relations
 \begin{equation}\label{eq:sl2}
  [\hh,\ee]=2\ee,\qquad
  [\hh,\ff]=-2\ff,\qquad
  [\ee,\ff]=\hh.
 \end{equation}
 Thus they define an $\slTwo$-action on $\cS$.

 \item
 For every integer $n\geq0$, the subspace
 \[
  V_n
  :=
  \operatorname{span}_{\C}
  \bigl\{P^j(\gamma_1^n)\mid 0\leq j\leq n\bigr\}
  \subseteq\cS
 \]
 is an irreducible $\slTwo$-module of highest weight $n$,
 with highest-weight vector $\gamma_1^n$.

 \item
 This $\slTwo$-action integrates uniquely to a rational
 representation of $SL_2(\C)$ on $\cS$.
\end{enumerate}
\end{theorem}

The new $SL_2(\C)$-action gives a representation theoretic
interpretation of the multiplicity spaces in
\eqref{eq:creutzig-decomposition}. Let
$\mathfrak h_{\mathrm{osc}}$ denote the Heisenberg algebra
generated by the nonzero modes of $h$ and the central element,
and let
\[
 \mathcal F_{\mathrm{osc}}
 =\C[h_{-1},h_{-2},\ldots]
\]
be its Fock representation, with commutation relations
$[h_r,h_s]=-2r\delta_{r+s,0}$.

\begin{theorem}\label{thm:joint-decomposition}
The $SL_2(\C)$-action of Theorem~\ref{thm:main} commutes with
the actions of the affine vertex algebra
$L_{-1}(\mathfrak{sl}_2)$ and the Heisenberg oscillator algebra
$\mathfrak h_{\mathrm{osc}}$ on $\cS$.
 There is an isomorphism
\begin{equation}\label{eq:intro-joint-decomposition}
 \cS
 \cong
 \bigoplus_{n\geq0}
 L_{-1}(n)\otimes V_n\otimes\mathcal F_{\mathrm{osc}}
\end{equation}
for these three mutually commuting actions, where $V_n$ is
the irreducible $SL_2(\C)$-module of highest weight $n$.

Under this identification, $h_0$ acts as the Cartan operator
$\hh$ on $V_n$. In particular, the weight decomposition of
$V_n$ recovers the Heisenberg Fock decomposition
\[
 V_n\otimes\mathcal F_{\mathrm{osc}}
 \cong
 \bigoplus_{j=0}^{n}\pi_{n-2j},
\]
and hence \eqref{eq:intro-joint-decomposition} recovers
the decomposition
\eqref{eq:creutzig-decomposition}.
\end{theorem}

The decomposition in Theorem~\ref{thm:joint-decomposition}
also specifies how the new action interacts with the conformal
grading. For the conformal structure in which each
$\beta_i,\gamma_i$ has weight $\frac12$, we have
\[
 L_0
 =
 L_0^{\mathrm{aff}}
 +N_{\mathrm{osc}}
 -\frac14\hh^2,
 \qquad
 N_{\mathrm{osc}}
 =
 -\frac12\sum_{r>0}h_{-r}h_r.
\]
Thus $L_0+\frac14\hh^2$ commutes with the new $SL_2(\C)$-action,
whereas $L_0$ itself does not. More explicitly,
\[
 [L_0,\ff]=\ff(\hh-1),\qquad
 [L_0,\ee]=-\ee(\hh+1).
\]
The new action therefore relates the different Heisenberg
charge sectors within each affine multiplicity space,
with the change in conformal weight determined by the
Heisenberg zero-mode contribution.

The note is organized as follows.
In Section~\ref{sec:lattice}, we explicitly construct the
lattice-type fields on a local chart and show that their residues
restrict to a well-defined $\mathfrak{sl}_2$-action on the rank-two
$\beta\gamma$-system $\cS$, proving the first part of
Theorem~\ref{thm:main}.
In Section~\ref{sec:prop}, we establish basic properties of this
action, construct finite-dimensional submodules, and prove that
it integrates to an $SL_2(\C)$-action, completing the proof of
Theorem~\ref{thm:main}. We then describe its action on the
multiplicity spaces in Creutzig's decomposition of $\cS$ and
prove Theorem~\ref{thm:joint-decomposition}.

\section{Construction of $\mathfrak{sl}_2$-action}\label{sec:lattice}

\subsection{A common localization}
Set
\begin{equation}\label{eq:polarization}
 x=\gamma_1,\qquad y=\beta_2,\qquad
 p_x=\beta_1,\qquad p_y=-\gamma_2.
\end{equation}
Then $p_x(z)x(w)\sim p_y(z)y(w)\sim(z-w)^{-1}$ and the two
coordinate fields $x,y$ have regular mutual OPE.  This is another
beta--gamma polarization.  In the injective localization
\[
 \widetilde{\cS}=\cS[x^{-1},y^{-1}],
\]
we have
\begin{equation}\label{eq:A-h}
 A=\frac{y}{x},\qquad S(A)=A^{-1},\qquad
 h=\norm{p_xx}-\norm{p_yy}.
\end{equation}

Define
\begin{equation}\label{eq:ell}
 \ell=\norm{A^{-1}\dd A}
      =\norm{y^{-1}\dd y}-\norm{x^{-1}\dd x}.
\end{equation}
All coordinate-field products in this formula lie in a commutative vertex
subalgebra.  In particular, both $\dd A=\norm{\ell A}$ and $\ell$ have regular
OPE with $A^{\pm1}$.

\begin{lemma}\label{lem:alpha}
On $\widetilde{\cS}$,
\begin{equation}\label{eq:hell}
 \begin{aligned}
 h(z)h(w)&\sim-\frac2{(z-w)^2},&
 h(z)\ell(w)&\sim-\frac2{(z-w)^2},\\
 \ell(z)\ell(w)&\sim0,&
 h(z)A(w)&\sim-\frac{2A(w)}{z-w}.
 \end{aligned}
\end{equation}
Moreover, $\ell_0=0$.  Consequently the field
\begin{equation}\label{eq:alpha}
 \alpha=\ell-h
\end{equation}
satisfies
\begin{equation}\label{eq:alpha-OPE}
 \alpha(z)\alpha(w)\sim\frac2{(z-w)^2},\qquad
 \alpha_0=-h_0.
\end{equation}
\end{lemma}
\begin{proof}
For one coordinate--momentum pair $p(z)t(w)\sim(z-w)^{-1}$, Wick
contraction gives
\[
 \norm{pt}(z)\norm{t^{-1}\dd t}(w)\sim\frac1{(z-w)^2}.
\]
Applying this separately to $(x,p_x)$ and $(y,p_y)$ proves the mixed
Heisenberg OPE, and the remaining OPEs in \eqref{eq:hell} follow directly.

To check the zero mode, we obtain
\[
 \norm{x^{-1}\dd x}(z)p_x(w)
       \sim\frac{x^{-1}(w)}{(z-w)^2},
\]
The same holds for $y$.  The OPEs with the coordinate
fields and their inverses are regular.  Since a zero mode is a derivation,
$\ell_0$ annihilates the entire localized vertex algebra.  Finally,
\eqref{eq:alpha-OPE} follows from \eqref{eq:hell} and $\ell_0=0$.
\end{proof}

\subsection{The charge-shifting operator}
\begin{lemma}\label{lem:U}
There is an invertible, $z$-independent operator $U$ on
$\widetilde{\cS}$ such that
\begin{equation}\label{eq:A-bosonized}
 A(z)=
 \exp\!\left(\sum_{n>0}\frac{\ell_{-n}}n z^n\right)
 U
 \exp\!\left(-\sum_{n>0}\frac{\ell_n}n z^{-n}\right).
\end{equation}
It satisfies
\begin{equation}\label{eq:U-charge}
 [\alpha_m,U]=2\delta_{m,0}U,\qquad
 [\alpha_m,U^{-1}]=-2\delta_{m,0}U^{-1}.
\end{equation}
\end{lemma}
\begin{proof}
Define $U(z)$ by multiplying $A(z)$ on the left and right by the inverses
of the exponentials in \eqref{eq:A-bosonized}.  All $\ell$-modes commute
with one another and with $A(z)$.  Because $\ell_0=0$ and
$\dd A=\norm{\ell A}$, differentiation shows $\dd_z U(z)=0$.
Thus $U(z)=U$ is independent of $z$.  Repeating the construction for
$A^{-1}$ gives its inverse, since the coordinate fields $A$ and $A^{-1}$
multiply to $1$ and commute with the exponential terms of $\ell$.

The OPEs imply
\[
 [h_m,A(z)]=-2z^mA(z),\qquad
 [h_m,\ell_n]=-2m\delta_{m+n,0}.
\]
For $m\ne0$ the commutator with the appropriate exponential cancels the
factor $-2z^m$, and hence $[h_m,U]=0$.  For $m=0$ one gets $[h_0,U]=-2U$.
Thus we deduce (\ref{eq:U-charge}), as all $\ell_m$ commute with $U$.  
\end{proof}

\subsection{The mixed operator product}
On the common localization define
\begin{equation}\label{eq:FG}
 F(z)=\Dr_h(A;z),\qquad G(z)=\Dr_{-h}(A^{-1};z),
\end{equation}
whose residues restrict to $P$ and $SPS$, respectively.  Substituting
\eqref{eq:A-bosonized} and using $\alpha=\ell-h$ gives
\begin{align}
 F(z)&=
 \exp\!\left(\sum_{n>0}\frac{\alpha_{-n}}n z^n\right)
 U z^{\alpha_0}
 \exp\!\left(-\sum_{n>0}\frac{\alpha_n}n z^{-n}\right),
 \label{eq:F-lattice}\\
 G(z)&=
 \exp\!\left(-\sum_{n>0}\frac{\alpha_{-n}}n z^n\right)
 U^{-1}z^{-\alpha_0}
 \exp\!\left(\sum_{n>0}\frac{\alpha_n}n z^{-n}\right).
 \label{eq:G-lattice}
\end{align}

We recall the definition in (\ref{eq:triple}) that $\ff=P=\Res_z F(z)$ and $\ee=SPS=\Res_z G(z)$.

\begin{proposition}\label{prop:mixed}
The fields in \eqref{eq:F-lattice}--\eqref{eq:G-lattice} are mutually local
and have singular parts
\begin{equation}\label{eq:mixed-OPE}
 \begin{aligned}
 F(z)G(w)&\sim\frac1{(z-w)^2}+\frac{\alpha(w)}{z-w},\\
 G(z)F(w)&\sim\frac1{(z-w)^2}-\frac{\alpha(w)}{z-w}.
 \end{aligned}
\end{equation}
\end{proposition}
\begin{proof}
Since $[\alpha_m,\alpha_n]=2m\delta_{m+n,0}$, commuting the positive-mode
exponential of $F(z)$ through the negative-mode exponential of $G(w)$
contributes
\[
 \exp\!\left(2\sum_{n>0}\frac{(w/z)^n}{n}\right)
 =(1-w/z)^{-2}
\]
in the $|z|>|w|$ formal expansion.  From
$z^{\alpha_0}U^{-1}=z^{-2}U^{-1}z^{\alpha_0}$, the zero modes contribute
$z^{-2}$.  Thus the scalar singular factor is $(z-w)^{-2}$.
After removing it, the regular factor is
\begin{align*}
 R(z,w)={}&
 \exp\!\left(\sum_{n>0}\frac{\alpha_{-n}}n(z^n-w^n)\right)
 (z/w)^{\alpha_0}\\
 &\quad\cdot
 \exp\!\left(-\sum_{n>0}\frac{\alpha_n}n(z^{-n}-w^{-n})\right).
\end{align*}
It satisfies $R(w,w)=\Id$ and
$\left.\dd_zR(z,w)\right|_{z=w}=\alpha(w)$.  This proves the first OPE.
Reversing the two charges gives the second.  The reversed radially ordered
products are the two expansions of the same expressions after
multiplication by $(z-w)^2$, which also proves mutual locality.
\end{proof}

\begin{lemma}\label{lem:P-preserves-S}
The operator $P$ preserves $\cS$.
\end{lemma}

\begin{proof}
Set
\[
 \cS_i:=\cS[\gamma_i^{-1}]\quad(i=1,2),
 \qquad
 \cS_{12}:=\cS[\gamma_1^{-1},\gamma_2^{-1}],
\]
and define
\[
 A':=-\norm{\gamma_2^{-1}\beta_1},\qquad
 C:=\gamma_1^{-1}\gamma_2^{-1}.
\]
A Wick calculation in $\cS_{12}$ gives
\begin{align*}
 \norm{hC}
 &=
 \norm{\gamma_2^{-1}\beta_1}
 +\norm{\gamma_1^{-1}\beta_2}
 -\gamma_1^{-2}\gamma_2^{-1}\dd\gamma_1
 -\gamma_1^{-1}\gamma_2^{-2}\dd\gamma_2\\
 &=A-A'+\dd C.
\end{align*}
On the other hand, differentiating
\eqref{eq:dressing-intro} yields
\[
 \dd_z\Dr_h(C;z)
 =
 \Dr_h\bigl(\dd C-\norm{hC};z\bigr).
\]
 Hence
\[
 \Dr_h(A;z)-\Dr_h(A';z)
 =-\dd_z\Dr_h(C;z).
\]
Taking residues, we obtain
\begin{equation}\label{eq:P-two-charts}
 P=\Res_z\Dr_h(A;z)=\Res_z\Dr_h(A';z).
\end{equation}

For $v\in\cS$, the two expressions in
\eqref{eq:P-two-charts} belong to $\cS_1$ and $\cS_2$,
respectively. The residues are well defined by truncation:
the positive Heisenberg modes act locally nilpotently on
$\cS$, and the remaining field expansions are lower
truncated on each vector. Thus
\[
 Pv\in\cS_1\cap\cS_2
 \qquad\text{inside }\cS_{12}.
\]

Finally, a fixed PBW ordering identifies the underlying
vector spaces of these localizations with the corresponding
localizations of the polynomial ring
\[
 R=\C[\partial^r\beta_i,\partial^r\gamma_i
       \mid i=1,2,\ r\geq0].
\]
Since $\gamma_1$ and $\gamma_2$ are independent polynomial
variables,
\[
 R[\gamma_1^{-1}]\cap R[\gamma_2^{-1}]
 =R
 \qquad\text{inside }R[\gamma_1^{-1},\gamma_2^{-1}].
\]
Therefore $\cS_1\cap\cS_2=\cS$, and hence $Pv\in\cS$.
\end{proof}

\begin{proof}[Proof of (1) in  Theorem~\ref{thm:main}]
Since the fields $F(z)$ and $G(z)$ are mutually local, the
formal-distribution commutator formula applied to the OPE
\eqref{eq:mixed-OPE} yields
\[
    [G_0,F_0]=-\alpha_0=h_0.
\]
Furthermore, \eqref{eq:U-charge} and $\alpha_0=-h_0$ give
\[
 [h_0,F(z)]=-2F(z),\qquad [h_0,G(z)]=2G(z).
\]
Taking residues proves \eqref{eq:sl2} in the common localization.

Since $P$ preserves $\cS$ by Lemma \ref{lem:P-preserves-S}, so does $SPS$, and hence the identities restrict to $\cS$.
\end{proof}

\section{Basic properties of the $\mathfrak{sl}_2$-action}\label{sec:prop}

\subsection{Finite-dimensional representations}\label{sec:strings}
%
The coordinate vector $x^n$
satisfies
\begin{equation}\label{eq:xn}
 \alpha_mx^n=0\ (m>0),\qquad \alpha_0x^n=-nx^n,
 \qquad U^nx^n=y^n.
\end{equation}
For the last equality, note that on polynomials in the commuting coordinate
fields $x,y$, the constant operator $U$ is multiplication by $y/x$.

\begin{proposition}\label{prop:endpoint}
For $n\geq0$,
\[
 P^nx^n=(-1)^{n(n-1)/2}n!\,y^n.
\]
\end{proposition}
\begin{proof}
The case $n=0$ is immediate.  For $n\geq1$, normal ordering the product of
$n$ copies of \eqref{eq:F-lattice}, and then applying it to $x^n$, gives
\begin{equation}\label{eq:nfold}
 \begin{split}
 F(z_1)\cdots F(z_n)x^n
 ={}&\left(\prod_{i=1}^n z_i^{-n}\right)
      \prod_{i<j}(z_i-z_j)^2\\
 &\quad\cdot
 \exp\!\left(\sum_{m>0}\frac{\alpha_{-m}}m
                       \sum_{i=1}^n z_i^m\right)y^n.
 \end{split}
\end{equation}
Here each same-charge oscillator contraction gives $(1-z_j/z_i)^2$,
and commuting $z_i^{\alpha_0}$ through the later copies of $U$ contributes
$z_i^2$ per pair, producing $(z_i-z_j)^2$.

Taking the $n$ residues asks for the coefficient of
$z_1^{n-1}\cdots z_n^{n-1}$ in the last two factors of \eqref{eq:nfold}.
The Vandermonde square already has total degree $n(n-1)$, exactly the
required total degree.  Only the constant term of the oscillator
exponential can contribute.  Thus
\begin{equation}\label{eq:vandermonde-coeff}
 P^nx^n=
 [z_1^{n-1}\cdots z_n^{n-1}]
 \prod_{i<j}(z_i-z_j)^2\,y^n.
\end{equation}
It remains to compute the relevant coefficient of the  square of Vandermonde determinant
$\det(z_i^{j-1})_{1\leq i,j\leq n}$.  
A term indexed by permutations
$\sigma,\tau\in S_n$ has the required exponents precisely when
$\tau(i)=n+1-\sigma(i)$ for every $i$.
Equivalently, $\tau=w_0\circ \sigma$, where $w_0$ is the reversing permutation such that
$w_0(i)=n+1-i.$
Therefore all contributing terms have the same sign, and the desired coefficient is
$
(-1)^{n(n-1)/2}n!.
$
This proves the formula.
\end{proof}

\begin{corollary}\label{cor:strings}
Set $v_j=P^j x^n$.  Then $v_0,\ldots,v_n$ are linearly independent,
$Pv_n=0$, and
\begin{equation}\label{eq:strings}
 \begin{aligned}
 \hh v_j&=(n-2j)v_j,\\
 \ff v_j&=v_{j+1}&& (0\leq j<n),\\
 \ee v_j&=j(n-j+1)v_{j-1}&& (1\leq j\leq n),\qquad
 \ee v_0=0.
 \end{aligned}
\end{equation}
Their span is the irreducible highest-weight $\slTwo$-module of dimension
$n+1$.
\end{corollary}
\begin{proof}
The original lattice-type formula shows $P(y^n)=0$, as its zero-mode factor is
$z^n$, the positive $h$-modes kill $y^n$, and $A(z)y^n$ has no negative
powers.  Since $Sx^n=y^n$, one also has $\ee x^n=SPy^n=0$.
Together with Proposition \ref{prop:endpoint}, we deduce that $v_n\ne0$ and $Pv_n=0$.
Hence the $\slTwo$-relations give \eqref{eq:strings}.  None of $v_0,\ldots,v_n$
can vanish, since otherwise $v_n$ would vanish, and their distinct
$h_0$-weights imply linear independence.  The irreducibility is also clear.
\end{proof}


\begin{lemma}
The actions of $P$ on the four generating states are
\begin{equation}\label{eq:generators}
 P\gamma_1=\beta_2,\qquad P\gamma_2=-\beta_1,
 \qquad P\beta_1=P\beta_2=0.
\end{equation}
\end{lemma}

\begin{proof}
We prove only the second identity, and the others follow by similar
calculations. Since the positive $h$-modes annihilate $\gamma_2$ and
$h_0\gamma_2=\gamma_2$, only two terms $\norm{A\gamma_2}$ and $-\norm{h\gamma_1^{-1}}$ contribute to the residue. Thus,
\begin{align*}
P\gamma_2
&=\norm{A\gamma_2}-\norm{h\gamma_1^{-1}}\\
&=\left(
    \norm{\gamma_1^{-1}(\norm{\beta_2\gamma_2})}
    +\dd\gamma_1^{-1}
  \right)
  -\left(
    \beta_1+\norm{\gamma_1^{-1}(\norm{\beta_2\gamma_2})}
    +\dd\gamma_1^{-1}
  \right)\\
&=-\beta_1.
\end{align*}
\end{proof}

\begin{proposition}\label{prop:notderivation}
The endomorphism $P$ is not a derivation of the vertex algebra $\cS$.
Consequently \eqref{eq:triple} is an action on the underlying vector space,
not an action by vertex-algebra derivations.
\end{proposition}
\begin{proof}
Since $\gamma_{1(0)}\gamma_2=0$, the derivation identity would imply
\[
 0=(P\gamma_1)_{(0)}\gamma_2+
       \gamma_{1(0)}(P\gamma_2)
   =\beta_{2(0)}\gamma_2-
       \gamma_{1(0)}\beta_1=2,
\]
which is a contradiction.
\end{proof}

\subsection{Integrability}
\label{sec:integrability}

We choose a Virasoro vector 
\[
L:=\frac{1}2 (\beta_1 \gamma_1'- \beta_1'\gamma_1)+\frac12(\beta_2 \gamma_2'-\beta_2' \gamma_2).
\]
Hence we equip $\mathcal S$ with the symmetric conformal grading, whose
grading operator $L_0$ satisfies
\[
\operatorname{wt}(\partial^r\beta_i)
=
\operatorname{wt}(\partial^r\gamma_i)
=
r+\frac12
\qquad (r\geq 0,\ i=1,2).
\]
Then
\[
\mathcal S
=
\bigoplus_{d\in\frac12\mathbb Z_{\geq0}}\mathcal S[d],
\qquad
\dim\mathcal S[d]<\infty.
\]

\begin{lemma}
We have the identity on $\cS$
\begin{align}
\label{eq:integrability-L0-P}
[L_0,P]=&P(h_0-1),\\
\label{eq:integrability-L0-E}
[L_0,\mathsf E]=&-\mathsf E(h_0+1).
\end{align}
\end{lemma}
\begin{proof}
Recall that, on the localization where $\gamma_1$ is invertible,
\[
P=\operatorname{Res}_z F(z),
\]
where
\[
F(z)=
\exp\left(-\sum_{n>0}\frac{h_{-n}}{n}z^n\right)
A(z)z^{-h_0}
\exp\left(\sum_{n>0}\frac{h_n}{n}z^{-n}\right),
\qquad
A=:\gamma_1^{-1}\beta_2:.
\]
The field $A$ has conformal weight zero. Hence
\[
[L_0,A(z)]=z\partial_z A(z),
\qquad
[L_0,h_n]=-nh_n.
\]
Applying these identities to the factors of $F(z)$, while
preserving their order, gives
\[
[L_0,F(z)]
=
z\partial_zF(z)+F(z)h_0.
\]
Indeed, the additional term compensates for
\[
z\partial_z(z^{-h_0})=-h_0z^{-h_0},
\qquad
[L_0,z^{-h_0}]=0.
\]
Taking residues and using
\[
\operatorname{Res}_z z\partial_zF(z)
=
-\operatorname{Res}_z F(z)
\]
proves \eqref{eq:integrability-L0-P}.
Conjugating it by $S$, we also obtain \eqref{eq:integrability-L0-E}
due to Theorem \ref{thm:main} and
the following relations
 \[
SL_0S=L_0,\qquad Sh_0S=-h_0.
\]
\end{proof}

\begin{theorem}\label{thm:integrability}
The operators $\mathsf E$ and $\mathsf F$ act locally nilpotently
on $\mathcal S$. Moreover, $\mathcal S$ is a direct sum of
finite-dimensional $\mathfrak{sl}_2$-modules. Consequently, this
action integrates uniquely to a rational representation of
$SL_2(\mathbb C)$ on $\mathcal S$.
\end{theorem}

\begin{proof}
We note that the operators $L_0$ and $h_0$ commute and are simultaneously
diagonalizable, and the eigenvalues of $h_0$ are integers.
Let $v\in\mathcal S$ be a simultaneous eigenvector with
\[
L_0v=dv,\qquad h_0v=mv.
\]
Using \eqref{eq:integrability-L0-P},
\eqref{eq:integrability-L0-E}, and the relations
\[
[h_0,\mathsf F]=-2\mathsf F,\qquad
[h_0,\mathsf E]=2\mathsf E,
\]
we obtain, by induction on $r\geq0$,
\begin{align}
h_0\mathsf F^rv
&=(m-2r)\mathsf F^rv,
&
L_0\mathsf F^rv
&=(d+rm-r^2)\mathsf F^rv,
\label{eq:integrability-F-degrees}\\
h_0\mathsf E^rv
&=(m+2r)\mathsf E^rv,
&
L_0\mathsf E^rv
&=(d-rm-r^2)\mathsf E^rv.
\label{eq:integrability-E-degrees}
\end{align}
Both $d+rm-r^2$ and $d-rm-r^2$ are negative for sufficiently
large $r$. Since $\mathcal S$ has no negative conformal degrees,
it follows that
\[
\mathsf F^rv=\mathsf E^rv=0
\qquad\text{for sufficiently large }r.
\]
Every vector of $\mathcal S$ is a finite sum of simultaneous
eigenvectors of $L_0$ and $h_0$. Thus $\mathsf E$ and
$\mathsf F$ are locally nilpotent on all of $\mathcal S$.
\end{proof}

To exhibit finite-dimensional invariant subspaces, define
\[
K=L_0+\frac14h_0^2.
\]
Equations \eqref{eq:integrability-F-degrees} and
\eqref{eq:integrability-E-degrees} imply
\[
[K,\mathsf E]=[K,\mathsf F]=[K,h_0]=0.
\]
Indeed, the eigenvalue $d+m^2/4$ is unchanged by either
$\mathsf E$ or $\mathsf F$.

The operator $K$ is diagonalizable, and every eigenspace
\[
\mathcal S^{K=\kappa}
=
\{v\in\mathcal S\mid Kv=\kappa v\}
\]
is an $\mathfrak{sl}_2$-submodule. If a simultaneous eigenvector
of $L_0$ and $h_0$ belongs to this eigenspace, then
\[
d+\frac{m^2}{4}=\kappa,
\]
so $0\leq d\leq\kappa$. Therefore
\[
\mathcal S^{K=\kappa}
\subseteq
\bigoplus_{\substack{d\in\frac12\mathbb Z_{\geq0}\\d\leq\kappa}}
\mathcal S[d],
\]
and the right-hand side is finite-dimensional. Consequently,
\[
\mathcal S
=
\bigoplus_{\kappa}\mathcal S^{K=\kappa}
\]
is a direct sum of finite-dimensional $\mathfrak{sl}_2$-modules.


\begin{proof}[Proof of (2) and (3) in  Theorem~\ref{thm:main}]
Parts (2) and (3) of Theorem 1.1 follow from Corollary 3.2 and Theorem 3.6, respectively.
\end{proof}

\subsection{Compatibility with Creutzig's decomposition}
\label{sec:creutzig-decomposition}

We explain how the $\slTwo$-action of
Theorem~\ref{thm:main} acts on the multiplicity spaces in
Creutzig's decomposition of $\cS$.
To distinguish the affine currents from the operators
$\ee,\ff,\hh$, write
\[
 E=\norm{\beta_2\gamma_1},\qquad
 F=\norm{\beta_1\gamma_2},\qquad
 H=\norm{\beta_1\gamma_1}-\norm{\beta_2\gamma_2}.
\]
These currents generate $L_{-1}(\mathfrak{sl}_2)$ and commute
with the Heisenberg vertex algebra $\pi$ generated by $h$.

For $n\geq0$, let $L_{-1}(n)$ denote the simple highest-weight
affine $\mathfrak{sl}_2$-module at level $-1$ whose top is the
irreducible $\mathfrak{sl}_2$-module of highest weight $n$.
For $s\in\mathbb Z$, let $\pi_s$ be the Heisenberg Fock module
generated by a vector $|s\rangle$ satisfying
\[
 h_0|s\rangle=s|s\rangle,\qquad
 h_r|s\rangle=0\quad(r>0).
\]
By \cite[Proposition~5.1]{C}, there is an isomorphism
\begin{equation}\label{eq:creutzig-branching}
 \cS
 \cong
 \bigoplus_{n\geq0}
 L_{-1}(n)\otimes M_n,
 \qquad
 M_n:=\bigoplus_{j=0}^{n}\pi_{n-2j},
\end{equation}
of $L_{-1}(\mathfrak{sl}_2)\otimes\pi$-modules.
The charges occurring in $M_n$ are precisely the weights of
the irreducible $\slTwo$-module of highest weight $n$.
The following commutation identities make the relation with
our construction explicit.

\begin{prop}\label{prop:commuting-right-action}
For every $m\in\mathbb Z$ and $x\in\{E,F,H\}$, we have
\[
 [x_m,\ee]=[x_m,\ff]=[x_m,\hh]=0.
\]
Moreover,
\[
 [h_r,\ee]=[h_r,\ff]=0\quad(r\neq0),
 \qquad
 [h_0,\ee]=2\ee,\qquad [h_0,\ff]=-2\ff.
\]
\end{prop}

\begin{proof}
Recall that $\ff=P=\Res_z\Dr_h(A;z)$, where
$A=\norm{\gamma_1^{-1}\beta_2}$, and set $B=\gamma_1^{-2}$.
A direct Wick calculation gives
\begin{align}\label{eq:affine-A-brackets}
 E(z) A(w)&\sim 0, \qquad H(z)A(w)\sim 0,\\
 F(z) A(w)&\sim \frac{B(w)}{(z-w)^2}+ \frac{\dd B(w)-\norm{hB}(w)}{z-w}. 
\end{align}
Since the affine currents commute with all modes of $h$,
they commute with the exponential factors in the lattice-type
operator. Differentiating \eqref{eq:dressing-intro}, we obtain
\begin{equation}\label{eq:dressing-derivative}
 \dd_z\Dr_h(B;z)
 =
 \Dr_h\bigl(\dd B-\norm{hB};z\bigr).
\end{equation}
Consequently, for every $m\in\mathbb Z$,
\begin{align*}
 [F_m,\Dr_h(A;z)]
 &=
 m z^{m-1}\Dr_h(B;z)
 +z^m\Dr_h\bigl(\dd B-\norm{hB};z\bigr)\\
 &=
 \dd_z\bigl(z^m\Dr_h(B;z)\bigr).
\end{align*}
Taking residues gives $[F_m,P]=0$. The corresponding
commutators with $E_m$ and $H_m$ vanish directly by
\eqref{eq:affine-A-brackets}.

Next, the identities
\[
 [h_\lambda A]=-2A,\qquad
 [h_r,h_s]=-2r\delta_{r+s,0}
\]
imply
\[
 [h_r,\Dr_h(A;z)]=0\quad(r\neq0),\qquad
 [h_0,\Dr_h(A;z)]=-2\Dr_h(A;z).
\]
Indeed, for $r\neq0$, the contribution from the lattice-type
exponentials cancels the contribution from $A(z)$.
Taking residues proves the assertions for $\ff$.

Finally, $S$ fixes $E,F,H$ and satisfies $Sh_rS=-h_r$.
The assertions for $\ee=SPS$ follow by conjugation.
Those for $\hh=h_0$ follow from the commutativity of the
affine and Heisenberg subalgebras.
\end{proof}

Let $\mathfrak h_{\mathrm{osc}}$ denote the Heisenberg algebra
generated by the nonzero modes $h_r$ and the central element,
with the central element acting so that
$[h_r,h_s]=-2r\delta_{r+s,0}$.
All Fock modules $\pi_s$ restrict to the same
$\mathfrak h_{\mathrm{osc}}$-module
\[
 \mathcal F_{\mathrm{osc}}
 :=
 \mathbb C[h_{-1},h_{-2},\ldots].
\]
Here the negative modes act by multiplication, while
\[
 h_r=-2r\frac{\partial}{\partial h_{-r}}
 \qquad(r>0).
\]

\begin{proposition}\label{prop:creutzig-sl2}
Under \eqref{eq:creutzig-branching}, the new $\slTwo$-action
acts on each multiplicity space $M_n$. There is an isomorphism
\[
 M_n\cong V_n\otimes\mathcal F_{\mathrm{osc}}
\]
of $\slTwo\oplus\mathfrak h_{\mathrm{osc}}$-modules, where
$V_n$ is the irreducible module in
Theorem~\ref{thm:main}. Consequently,
\begin{equation}\label{eq:joint-decomposition}
 \cS
 \cong
 \bigoplus_{n\geq0}
 L_{-1}(n)\otimes V_n\otimes\mathcal F_{\mathrm{osc}}
\end{equation}
as a module for the mutually commuting affine
$\mathfrak{sl}_2$-action, new $\slTwo$-action, and
$\mathfrak h_{\mathrm{osc}}$-action.
Under this identification, $h_0$ acts as $\hh$ on $V_n$.
\end{proposition}

\begin{proof}
By Proposition~\ref{prop:commuting-right-action}, the new $\slTwo$-action
commutes with the affine algebra and hence acts on its
multiplicity spaces $M_n$.

The vector $\gamma_1^n$ satisfies
\[
 E_0\gamma_1^n=0,\qquad
 H_0\gamma_1^n=n\gamma_1^n,\qquad
 x_r\gamma_1^n=0
 \quad(x\in\{E,F,H\},\ r>0),
\]
and
\[
 h_0\gamma_1^n=n\gamma_1^n,\qquad
 h_r\gamma_1^n=0\quad(r>0).
\]
Thus it determines the Heisenberg highest-weight line of
charge $n$ in $M_n$.

Set
\[
 v_j:=P^j(\gamma_1^n),\qquad 0\leq j\leq n.
\]
Proposition \ref{prop:commuting-right-action} implies that each $v_j$
is an affine highest-weight vector of highest weight $n$
and a Heisenberg highest-weight vector of charge $n-2j$.
By Theorem~\ref{thm:main}, these vectors are nonzero and satisfy
\begin{equation}\label{eq:multiplicity-sl2}
 \hh v_j=(n-2j)v_j,\qquad
 \ff v_j=v_{j+1},\qquad
 \ee v_j=j(n-j+1)v_{j-1},
\end{equation}
where $v_{-1}=v_{n+1}=0$.

Each Fock summand in $M_n$ has a one-dimensional
Heisenberg highest-weight space. Hence the vectors
$v_0,\ldots,v_n$, viewed in the affine multiplicity space,
span all these highest-weight lines.
Since $\ee$ and $\ff$ commute with the negative Heisenberg
modes, the map
\[
 V_n\otimes\mathcal F_{\mathrm{osc}}\longrightarrow M_n,
 \qquad
 v_j\otimes p\longmapsto p\,v_j
\]
is an isomorphism of
$\slTwo\oplus\mathfrak h_{\mathrm{osc}}$-modules.
Combining these isomorphisms with
\eqref{eq:creutzig-branching} proves
\eqref{eq:joint-decomposition}.
\end{proof}

Thus the decomposition into Heisenberg Fock modules is
recovered from the weight decomposition of $V_n$:
\[
 V_n\otimes\mathcal F_{\mathrm{osc}}
 =
 \bigoplus_{j=0}^{n}
 \mathbb C v_j\otimes\mathcal F_{\mathrm{osc}}
 \cong
 \bigoplus_{j=0}^{n}\pi_{n-2j}.
\]
In particular, $\ff$ and $\ee$ connect adjacent Fock summands
while preserving their oscillator structure. They are not
homomorphisms for the full Heisenberg algebra, since they
change the eigenvalue of $h_0$.

\begin{remark}\label{rem:conformal-grading}
The new $\slTwo$-action does not preserve the conformal grading
of \cite{C}, since
\[
 \operatorname{wt}(v_j)
 =\frac{n(n+2)-(n-2j)^2}{4}
 =\frac n2+j(n-j).
\]
However, $K=L_0+\frac14h_0^2$ commutes with this action and
acts on $V_n$ by the scalar $\frac{n(n+2)}4$; see also the
discussion following the proof of
Theorem~\ref{thm:integrability}.
\end{remark}

\end{document}